\documentclass{amsart}
 
 \usepackage{amscd}
 \usepackage{amssymb}
  \usepackage[curve,matrix,arrow]{xy}

\usepackage{amsthm}

\newtheorem{thm}[equation]{Theorem}
\newtheorem{lem}[equation]{Lemma}
\newtheorem{cor}[equation]{Corollary}
\newtheorem{prop}[equation]{Proposition}

\newtheorem*{thm*}{Theorem}
\newtheorem*{prop*}{Proposition}
\newtheorem*{cor*}{Corollary}
\newtheorem*{lem*}{Lemma}
\newtheorem*{MT*}{Main Theorem}

\theoremstyle{definition} %

\newtheorem*{defn*}{Definition}

\newtheorem{eg}[equation]{Example}

\theoremstyle{remark} %
\newtheorem{rmk}[equation]{Remark}

\newtheorem*{rmk*}{Remark}
\newtheorem*{rmks*}{Remarks}

\newtheoremstyle{exercise}
  {3pt}
  {3pt}
  {\small}
  {}
  {\sc\small}
  {.}
  {.5em}
   {}     
  {}

\theoremstyle{exercise}

\makeatletter
{\renewcommand{\theequation}{#1}}%
{\renewcommand{\theequation}{\arabic{equation}}\addtocounter{equation}{-1}\global\@ignoretrue}
\makeatother

\makeatletter
{\renewcommand{\theequation}{#1}\begin{eqnarray}}%
{\end{eqnarray}\renewcommand{\theequation}{\arabic{equation}}\addtocounter{equation}{-1}\global\@ignoretrue}
\makeatother

\makeatletter
{\smallskip \refstepcounter{equation}\noindent{\textbf{\theequation.} }{{\textbf{#1.}}}}%
{\smallskip \global\@ignoretrue}
\makeatother

\makeatletter
{\smallskip \refstepcounter{equation}{\sc \theequation}{\sc (#1).}}%
{\smallskip \global\@ignoretrue}
\makeatother

\makeatletter
{\smallskip \refstepcounter{equation}\noindent{\sc \theequation.}{\sl{ #1.}}}%
{\smallskip \global\@ignoretrue}
\makeatother

\makeatletter
\newenvironment{borel*}%
{\smallskip \refstepcounter{equation}\noindent{\textbf{\theequation.}}}%
{\global\@ignoretrue}
\makeatother

\newcommand{\flist}[1]{\hangindent\leftmargini\textup{(1)}\hskip\labelsep {#1}%
\begin{enumerate}%
\setcounter{enumi}{1}%
}

\newcommand{\ot}{\otimes}

\newcommand{\darkrad}{0.115}
\newcommand{\lrad}{0.25}

\newcommand{\eand}{\quad\text{and}\quad}

\newcommand{\Z}{{\mathbb{Z}}}        

\newcommand{\Zm}[1]{\Z/{#1}\Z}

\newcommand{\la}{\lambda}
\newcommand{\D}{\Delta}

\newcommand{\proj}{{\mathbb{P}}}     

\newcommand{\oddots}{{\mathinner{\mkern1mu\raise1pt\vbox{\kern7pt\hbox{.}}\mkern2mu\raise4pt\hbox{.}\mkern2mu\raise7pt\hbox{.}\mkern1mu}}}

\newcommand{\s}{\sigma}

\newcommand{\As}{{(A,\sigma)}}

\newcommand{\Bt}{{(B,\tau)}}

\newcommand{\Fx}{{F^{\times}}}
\newcommand{\Fxsq}{{F^{\times2}}}

\newcommand{\Fsq}{{\Fx/\Fxsq}}

\newcommand{\qform}[1]{{\left\langle{#1}\right\rangle}}                   

\newcommand{\basemu}{\boldsymbol{\mu}}
\newcommand{\mmu}[1]{\basemu_{#1}}     

\DeclareMathOperator{\Spin}{Spin}           
\newcommand{\Sp}{\mathrm{Sp}}
\DeclareMathOperator{\SL}{SL}

\newcommand{\SO}{\mathrm{SO}}
\newcommand{\SU}{\mathrm{SU}}

\DeclareMathOperator{\Ad}{Ad}

\DeclareMathOperator{\Int}{Int}

\DeclareMathOperator{\disc}{disc}

\DeclareMathOperator{\Br}{Br}

\newcommand{\Nrd}{{\mathrm{Nrd}}}

\newcommand{\stbtmat}[4]{\left( \begin{smallmatrix} #1&#2 \\ #3&#4 \end{smallmatrix} \right) }

 \numberwithin{equation}{section}
 
\numberwithin{figure}{section}
\numberwithin{table}{section}

\newcommand{\Falg}{F_{{\mathrm{alg}}}}
 
\renewcommand{\D}{\mathcal{D}}

\DeclareMathOperator{\Sym}{Sym}
\DeclareMathOperator{\ad}{ad}

\newcommand{\binv}{{\bar{\ }}}
 
\begin{document}
 \title[Pfister's Theorem for involutions of degree 12]{Pfister's Theorem for orthogonal involutions of degree 12}
 
 \author{Skip Garibaldi}
\address{Department of Mathematics \& Computer Science, Emory University, Atlanta, GA 30322, USA}
\email{skip@member.ams.org}
\urladdr{http://www.mathcs.emory.edu/{\textasciitilde}skip/}

\author{Anne Qu\'eguiner-Mathieu} 
\address{
Universit\'e Paris 13 (LAGA) \\
CNRS (UMR 7539)\\
Universit\'e Paris 12 (IUFM) \\
93430 Villetaneuse\\
France}
\email{queguin@math.univ-paris13.fr}
\urladdr{http://www-math.univ-paris13.fr/{\textasciitilde}queguin/}



\setlength{\unitlength}{.75cm}

\begin{abstract}
We use the fact that a projective half-spin representation of $\Spin_{12}$ has an open orbit to
 generalize Pfister's result on quadratic forms of dimension 12 in
$I^3$ to orthogonal involutions.
\end{abstract}

\maketitle

In his seminal paper \cite{Pfister}, Pfister proved strong theorems
describing quadratic forms of even dimension $\le 12$ that have
trivial discriminant and Clifford invariant, i.e., that are in $I^3$.
His results have been extended to quadratic forms of dimension 14 in
$I^3$ by Rost, see \cite{Rost:14.1} or \cite{G:lens}.  One knows also
extensions of these theorems where quadratic forms are replaced by 
central simple algebras with orthogonal involution, except in degree
$12$.  The purpose of this paper is to complete this picture by giving
the extension in the degree 12 case.  The principle underlying the quadratic forms results and our extension is that a projective half-spin representation of $\Spin_n$ for even $n$ has an open orbit precisely for $n \le 14$, cf.~\cite{Rost:14.1}, \cite{Igusa}, and \cite{SK}.

Let us first recall what is already known. We consider quadratic forms
$q$ of even dimension (resp.~central simple algebras with orthogonal
involution $\As$ of even degree) with trivial discriminant and
Clifford invariant.  If $q$ has dimension $< 8$, then $q$ is
hyperbolic by the Arason-Pfister Hauptsatz \cite[X.5.1]{Lam}.
This also holds in the non split case: if $A$ has degree $< 8$, then
$\s$ is hyperbolic, see e.g.~\cite[1.4]{G:16} or~\cite[4.4]{Q:03}.   If $q$ has dimension 8, then $q$ is similar to a 3-Pfister form \cite[X.5.6]{Lam}; if $A$ has degree 8, then $\As$ is isomorphic to a tensor product $\ot_{i=1}^3 (Q_i, \s_i)$ of quaternion algebras with orthogonal involution \cite[42.11]{KMRT}.   If $A$ has degree 10 or 14, then $A$ is necessarily split \cite[1.5]{G:16}, so there is no interesting generalization of the theorem on quadratic forms.  The remaining case is where $q$ has dimension 12, where Pfister proved that $q$ is isomorphic to $\phi \ot \psi$ for some 1-Pfister $\phi$ and 6-dimensional form $\psi$ with trivial discriminant, see \cite[pp.~123, 124]{Pfister} or \cite[17.13]{G:lens}.  In Theorem \ref{complete} below, we prove an analogous statement for $\As$ in case $A$ has degree 12.  We do not use Pfister's theorem on 12-dimensional quadratic forms in our proof, so we obtain his result as a corollary (Corollary \ref{Pfister}).


\smallskip
The paper comes in three sections. 
First, we study quadratic extensions of algebras with involution, and
in particular orthogonal extensions of unitary involutions, as
considered in~\cite[App.~2]{BP}, \cite[\S3]{ElomaryTig}, and
\cite[2.14]{QT}.  Second, we show how to construct an algebra
with orthogonal involution that has trivial discriminant and Clifford invariant from
any exponent $2$ algebra with unitary involution. 
Third, 
we prove that in degree $12$, this construction produces every central
simple algebra with orthogonal involution that has trivial discriminant and Clifford invariant.


In the language of linear algebraic groups, our results are as follows.  Our construction takes a group of type $^2\!A_{n -1}$ (e.g., $\SU_n$) and produces a group of type $^1\!D_n$ (e.g., $\Spin_{2n}$).  Our Theorem \ref{complete} shows that, over a field $F$, every algebraic group of type $D_6$ with a half-spin representation defined over $F$ is obtained by our construction.

\subsection*{Global conventions.}  We work over a base field $F$ of
characteristic $\ne 2$. 
Typically, we use the notation and basic language of \cite{KMRT} without comment.  
A few main points are: We extend the notions of central simple algebra and Brauer-equivalence over a quadratic field extension $K/F$ in an obvious way to include also the case where $K$ is the split \'etale algebra $F \times F$.
For $a \in A^\times$, we write $\Int(a)$ for the map $x \mapsto axa^{-1}$ on $A$.

Let $\As$ be a central simple algebra of even degree with orthogonal
involution.  The (signed) \emph{discriminant} of $\s$ is an element of
$\Fsq = H^1(F, \mmu2)$.  If the discriminant of $\s$ is trivial, then
the even Clifford algebra of $\As$ is a product $C_+ \times C_-$ of
central simple algebras such that $[C_+] - [C_-]$ equals $[A]$ in the
Brauer group of $F$ (see for instance~\cite[(9.12)]{KMRT}).  The
\emph{Clifford invariant} is the class of $[C_+]$ or $[C_-]$ in the
quotient $\Br(F) / [A]$. 

Following Becher, for any $n$-dimensional quadratic form $q$ over
$F$, we write $\Ad_q$ for the split algebra with
adjoint involution $(M_n(F),\ad_q)$. 
 
\section{Quadratic extensions of algebras with involution}

\begin{borel*} \label{DtoA}
Let $\As$ be a central simple $F$-algebra with involution,
and suppose that $A$ contains a $\sigma$-stable quadratic \'etale
$F$-algebra $K$. 
One can write $K$ as $F[\delta]/(\delta^2-d)$ for some $d \in
\Fx$, so this hypothesis is equivalent to saying: there is a
non-central $\delta
\in A$ such that $\s(\delta) = \pm\delta$ and $\delta^2 \in \Fx$.  
The centralizer $B$ of $K$ in $A$ is a central simple
$K$-algebra Brauer-equivalent to $A \ot K$, and of degree $\frac12
\deg(A)$. We write $\tau$ for the involution on $B$ induced by $\s$,
and 
we say that $(A,\sigma)$ is a {\em quadratic
  extension} of $(B,\tau)$. 
Note in particular that if $A$ and $B$ are division, $A$ is a
quadratic internal extension of $B$ as in Dieudonn\'e~\cite{Dieu:ext}. 

\begin{eg}
\label{quat.eg}
Let $(Q,\gamma)$ be a quaternion $F$-algebra, with involution of the first
kind. For any
$\gamma$ symmetric or skew-symmetric pure quaternion $i$, $(Q,\gamma)$
is a quadratic extension of $(F(i),\gamma_{|F(i){}})$. 
\end{eg} 

\begin{eg}
\label{descent.eg}
Consider now another algebra with involution of the first kind
$(B_0,\tau_0)$ over $F$; the tensor product $\As=(B_0,\tau_0)\ot(Q,\gamma)$ is
a quadratic extension of
$(B,\tau)=(B_0,\tau_0)\ot(F(i),\gamma_{|F(i){}})$. 
\end{eg}

\begin{eg} In particular, if a quadratic form $q$ decomposes as 
$\psi\ot\qform{1,-d}$, the algebra with involution
$\Ad_q$ is
  a quadratic extension of
  ${\Ad_\psi}\otimes_F(F(\sqrt{d}),\binv)$. 
\end{eg} 

\begin{borel*} \label{image} 
We will now restrict our attention to extensions of $K/F$-unitary
involutions, that is we assume $K=F[\delta]$ with $\delta^2=d$ and
$\s(\delta)=-\delta$. 
If $A$ is not split, $\s$ is not orthogonal, or $\s$ has even Witt index (e.g., if $\s$ is anisotropic), then by \cite[3.3]{BST}, the existence of such a $\delta$ is equivalent to the statement \emph{$\s$ is hyperbolic over $F$ or a quadratic extension of $F$}.
\end{borel*}

One remarkable fact about the situation in \ref{DtoA} is that the involution $\s$ is uniquely determined by its
restriction $\tau$: 
\begin{lem} \label{unique}
Let $(A,\s)$ be a quadratic extension of $(B,\tau)$, and assume that $\tau$ is
$K/F$ unitary. The only involutions on $A$ that agree with $\tau$
on $B$ are $\s$ and $\Int(\delta)\circ \s$. In particular, there is a
unique orthogonal (resp.~symplectic) involution on $A$ that is 
$\tau$ on $B$.  
\end{lem}
\begin{proof} If $\tau$ is $K/F$ unitary, then in particular it acts
  on $F$ as the identity, and every involution on $A$ that agrees with $\tau$
  on $B$ is of the first kind. Let $\s'$ be such an involution, and
  pick $a \in A^\times$ such that $\s(a) = \pm a$ and $\s' = \Int(a)
  \circ\s$.  
Since $\s'$ and $\s$ agree on $B$, the element $a$ centralizes $B$ and
  by the Double Centralizer Theorem belongs to $K$.  If $\s(a) = a$,
  this means $a$ is in $F$ and $\s' = \s$. Otherwise, $a\in \delta F$
  and $\s'=\Int(\delta)\circ\s$ and the type of $\s'$ is opposite that of $\s$.
  \end{proof}
\begin{rmk} 
It is obvious, both from the proof above and from
  Example \ref{quat.eg} that this is not true anymore if $\tau$ is
  of the first kind.
\end{rmk}
\end{borel*}

\begin{borel*} \label{QT}
Lemma~\ref{unique} shows that $\s$ is completely determined by its
restriction to $B$. We now describe an explicit procedure to recover
$\As$ from $\Bt$ if we know the Brauer class of $A$ and the type of $\s$, thus proving: 
\end{borel*}

\begin{prop} \label{existence.prop}
Let $K$ be a quadratic \'etale $F$-algebra and $\Bt$ a
  central simple $K$-algebra with unitary $K/F$-involution. 
We assume $B$ has exponent $2$, and we
fix a Brauer class $\beta\in\Br_2(F)$ whose restriction to $K$ is the
class of $B$.
There exists a unique orthogonal (resp.~symplectic) quadratic extension $\As$ of $\Bt$ with Brauer
class $\beta$. 
\end{prop} 

\begin{proof} 
We prove existence, beginning as in \cite[pp.~380, 381]{ElomaryTig}:
Since $B$ has
exponent $2$, it has an involution $\nu$ of the first kind of the desired type (orthogonal or symplectic).   One can find an element $u \in B^\times$ such that 
\begin{equation} \label{QT.1}
\nu(u) = \tau(u) = u \eand (\nu\tau)^2(x) = uxu^{-1} \quad (x \in B).
\end{equation}
We define an $F$-algebra $A_1$ to be the vector space $B \oplus Bz_1$ with multiplication rules
\[
z_1^2 = u \eand z_1b = (\nu\tau )(b)\,z_1.
\]
It is a central simple $F$-algebra \cite[Chap.~11, Th.~10]{Albert}. 
Moreover, the centralizer of $K$ in $A_1$ is $B$. 
So the Brauer class of $A_1\ot_F K$
is the class of $B$, that is $\beta$ extended to $K$, and there is some choice of $\la\in\Fx$ such that $A_1 \ot (K, \la)$ has Brauer-class $\beta$.  

Equations \eqref{QT.1} only determine $u$ up to a central factor. By \cite[13.41]{KMRT} replacing $u$ by $\la u$ in the
construction above produces an algebra $A=B\oplus Bz$ with Brauer class $\beta$.

This algebra $A$ is endowed with an involution $\s$ defined by
\begin{equation} \label{s.def}
\s(b_1 + b_2 z) := \tau(b_1) + z \tau(b_2) \quad (b_1, b_2 \in B).
\end{equation}
We show that this involution has the same type as $\nu$.
Write $\Sym\As$ and $\Sym\Bt$ for the subspaces of symmetric elements.  Clearly,
\[
\Sym \As = \Sym \Bt \oplus z \Sym (B, \nu).
\]
So the dimension over $F$ of $\Sym \As$ is 
\[
\dim_F \Sym \As = m^2 + m (m + \varepsilon) = \frac{2m (2m + \varepsilon)}{2},
\]
where $m$ is the degree of $B$, and 
$\varepsilon=\pm 1$ the type of $\nu$. 
This proves that $\s$ also is of type $\varepsilon$.

By Lemma \ref{unique}, the isomorphism class of $\As$ depends only on $\Bt$, $\beta$, and the type of $\s$,  and not on  the particular choices of $u$ and $\nu$.
\end{proof} 

\begin{eg}
If $\Bt$ is hyperbolic (e.g., if $K=F\times F$), then $B$ contains an
idempotent $e$ satisfying $\tau(e)=1-e$; the same equalities also hold
in $A$, so $\As$ is hyperbolic \cite[2.1]{BST}.
\end{eg}

\begin{eg} \label{descent} 
Assume that $(B,\tau)$ has a descent, that is $B=B_0\otimes_F K$ and 
$\tau=\tau_0\otimes\binv$ for some central simple algebra $B_0$ over $F$ 
with involution $\tau_0$ of the first kind.  (For example, this holds if $B$ is split.)
Pick $\la\in\Fx$ such that $\beta$ is the Brauer class of $B_0\ot (K,\lambda)$. 
The algebra with involution $(B,\tau)$ has, up to isomorphism, exactly two extensions with Brauer
class $\beta$, one of orthogonal type and one of symplectic type. 
They both decompose as
\begin{equation} \label{descent.2} 
(A,\sigma)=(B_0,\tau_0)\otimes((K,\lambda),\gamma),
\end{equation}
where either $\gamma$ is the only orthogonal involution
on $(K,\lambda)$ that acts as $\binv$ on $K$, and $\s$ is of
the same type as $\tau_0$, or $\gamma$ is the canonical
involution on $(K,\lambda)$, and $\s$ and $\tau_0$ are of opposite
type. 
This follows directly from Example~\ref{descent.eg} and Proposition~\ref{existence.prop}. 
\end{eg}

\begin{rmk}
We now sketch the relationship between $\Bt$ and $\As$ from the perspective of algebraic groups.  Write $G$ for the group $\SO\As$ if $\s$ is orthogonal, resp.~$\Sp\As$ if $\s$ is symplectic, and suppose that $K = F[\delta]$ is a field.  Over a separable closure of $F$, the eigenspaces $V, W$ of $\delta$ are parallel totally isotropic subspaces.  The stabilizers $P, Q$ of $V, W$ in $G$ are maximal parabolic subgroups defined over $K$.  Their intersection consists of the elements of $G$ that stabilize both $V$ and $W$, i.e., that commute with $\delta$.  It follows that the intersection $L := P \cap Q$ is an $F$-defined subgroup of $G$.  Its $F$-points are 
\[
L(F) = \{ b \in B \mid \text{$\tau(b)b = 1_B$ and $N_{K/F}(\Nrd_B(b)) = 1_F$} \}.
\]
The group $L$ is reductive with center the norm 1 elements of $K^\times$ and derived subgroup $\SU\Bt$.
Obviously, $L$ is a Levi subgroup of $P$ and $Q$, hence $P$ and $Q$ are opposite parabolics.  Moreover, the nontrivial $F$-automorphism of $K$ interchanges the eigenspaces of $\delta$ hence also $P$ and $Q$. 

In this section, we have described how to construct $\As$ from $\Bt$; finding $\Bt$ in $\As$ is a triviality given $K$.  The second, easier direction is also standard from the viewpoint of algebraic groups:  If we assume that $\s$ is hyperbolic over a quadratic extension $K/F$, then \cite[p.~383, Lemma 6.17$'$]{PlatRap} gives the existence of opposite parabolic subgroups $P$ and $Q$ over $K$  like those in the previous paragraph. (But note that this result requires $P$ and $Q$ to be conjugate, so it does not apply in case $\s$ is orthogonal and $\deg A \equiv 2 \bmod{4}$.)
\end{rmk}

\section{Orthogonal extensions of $\Bt$ and their invariants}  

From now on, we consider algebras with {\em orthogonal} involution
$\As$ that are quadratic extensions of some unitary $\Bt$. 
Suppose further that $B$ has even degree, so that 
the degree of $A$ is divisible by $4$. 

\begin{borel*}  \label{invts}
We claim that \emph{the
  discriminant of $\s$ is trivial} and \emph{the class of the
  discriminant algebra of $\Bt$ and the Clifford invariant of $\As$
  agree in $H^2(F, \mu_2)/[A]$.}  It suffices to check these claims
when $A$ is split, as can be seen by extending scalars to the function
field of the Severi-Brauer variety of $A$.  In that case, $\As$ is
given by \eqref{descent.2} where $B_0$ is split, and $\tau_0$ and
$\gamma$ are orthogonal.  The discriminant of $\s$ is obviously
trivial.  Further, the even Clifford algebra of $\As$ is $(K, \disc
\tau_0)$ by \cite[V.3.15, V.3.16]{Lam} or \cite[p.~147]{KMRT}, and the discriminant algebra of $(B, \tau)$ is also Brauer-equivalent to $(K, \disc \tau_0)$ by \cite[10.33]{KMRT}.
\end{borel*}

Hence, $(A,\s)$ has trivial discriminant and Clifford invariant if and only if 
the discriminant algebra of $(B,\tau)$ is split or Brauer equivalent
to $A$. We have proved:
\begin{lem} \label{rho.lem}
For any even degree algebra $\Bt$ with unitary $K/F$ involution, the 
unique orthogonal quadratic extension of $\Bt$ Brauer equivalent to
the discriminant algebra $\D\Bt$ has trivial discriminant and
Clifford invariant.   $\hfill\qed$
\end{lem}

The orthogonal quadratic extensions given by Lemma \ref{rho.lem} can have index 1, 2, or 4, and all three are possible.  Indeed, the discriminant algebra $\D\Bt$ has index dividing 4 \cite[10.30]{KMRT}.  Examples where $B$ has exponent 2 (as in \ref{existence.prop}) and $\D\Bt$ has index 4 can be found in
\cite[p.~145, Exercise 13]{KMRT} or can easily be constructed from Example \ref{descent}, even over some fields of cohomological dimension $\le 2$.  

\begin{eg} \label{quad.split} 
Suppose that $B$ is split and has even degree $n$, and the
discriminant algebra $\D\Bt$ is also split.  Then $\Bt$ descends as in
Example \ref{descent}, with $\tau_0$ adjoint to a quadratic form $\psi_0
\cong \qform{\alpha_1, \alpha_2, \ldots, \alpha_n}$.  As the
discriminant algebra is split, the discriminant of $\psi_0$ is a norm
from $K$.  So the form $\psi  := \qform{ (\disc \psi_0) \alpha_1, \alpha_2,
  \ldots, \alpha_n}$ is also a descent of $\tau$, i.e., $\ad_{\psi} \ot\,
\binv$ is an involution on $M_n(F) \ot K$ isomorphic to $\tau$.  Thus, the orthogonal quadratic
extensions of $(B,\tau)$ can be decomposed as 
\begin{equation} \label{quad.1}
\As = {\Ad_\psi} \ot (Q, \gamma)
\end{equation}
where $\psi$ has trivial discriminant, $Q$ is a quaternion algebra split
by $K$, and $\gamma$ is the unique orthogonal involution on $Q$ whose restriction to $K$ is $\binv$.

By \ref{invts}, those $\As$ have trivial discriminant and Clifford invariant.  As the algebra has index 2, there is an Arason/$e_3$ invariant defined for it via the method of \cite{Berhuy:quat} or either of the methods in \cite{G:16}.  We find:
\begin{equation}
e_3 \As = [K] \cdot e_2(\psi) \quad \in H^3(F, \Zm2) / \qform{[Q]},
\end{equation}
for $\psi$ as in \eqref{quad.1} and where $[K]$ denotes the class of $K$
in $H^1(F, \mmu2)$ and $e_2(\psi)$ is the Clifford invariant of $\psi$.  To
prove this formula, one need only check it when $Q$ is split.  In that
case, ${\ad_\psi}\ot\gamma$ is adjoint to the quadratic form $\psi$ tensored
with the norm form of the quadratic extension $K/F$, and the formula is clear.
\end{eg}

\section{Generalization of Pfister's theorem}

From any central simple algebra $\Bt$ of even degree with unitary
involution, \ref{rho.lem} produces a central simple algebra
of degree $2n$ with trivial discriminant and Clifford invariant. 
We assert that for $n = 6$, this construction produces all such algebras of
degree 12. 

\begin{thm} \label{complete}
Let $\As$ be a central simple algebra with orthogonal involution,
where $A$ has degree $12$.  If $\s$ has trivial discriminant and
Clifford invariant, then $\As$ is a quadratic extension of some central simple algebra $\Bt$  of
degree $6$ and exponent $2$ with unitary involution and such that the
discriminant algebra $\D\Bt$ is Brauer-equivalent to $A$.
\end{thm}

It is an obvious corollary of the theorem that the central simple algebra $\As$ is split by an extension of degree dividing 4.  Roughly speaking, this says that the torsion index of the half-spin group in dimension 12 is 4, a result of Totaro's \cite[5.1]{Tot:E8tor}.

\begin{proof}[Proof of Th.~\ref{complete}]
If $A$ is split, then $\s$ is adjoint to a 12-dimensional quadratic form $q$ in $I^3$, which cannot have odd Witt index by Pfister's theorem for 10-dimensional quadratic forms in $I^3$ and the Arason-Pfister Hauptsatz.  Therefore, by \ref{image} it suffices to prove that $\s$ is hyperbolic over $F$ or a quadratic extension of $F$.  We may assume that $F$ is infinite.

As $\As$ has trivial discriminant and Clifford invariant, one of the half-spin representations $V$ of $\Spin\As$ is defined over $F$.  Over an algebraic closure $\Falg$, the group $\Spin\As$ has an open orbit $\mathcal{O}$ in $\proj(V)(\Falg)$, see \cite[p.~1012]{Igusa} or \cite{G:lens}.  But $F$ is infinite, so there is some $v \in V$ \emph{over $F$} such that $[v]$ belongs to $\mathcal{O}$.

We consider the image in $\SO\As$ of the stabilizer of $[v]$ in $\Spin\As$.  Its identity component $H$ is isomorphic over $\Falg$ to $\SL_6$ by \cite{Igusa}.  Indeed, over $\Falg$, we may identify $\SO\As$ with the special orthogonal group $\SO_{12}$ of the symmetric bilinear form 
\[
(x, y) \mapsto x^t \stbtmat{0}{1_6}{1_6}{0} y
\]
where $1_6$ denotes the 6-by-6 identity matrix.  As $[v]$ belongs to the orbit $\mathcal{O}$, up to conjugacy $H$ is the copy of $\SL_6$ in $\SO_{12}$ given by the inclusion
\[
a \mapsto \stbtmat{a}{0}{0}{a^{-t}}.
\]
From this, we see that the natural 12-dimensional representation of
$\SO_{12}$ decomposes as a direct sum of inequivalent 6-dimensional
representations of $\SL_6$.  The centralizer $C$ of $H$ in
$\SO\As$ is a rank $1$ torus; it consists over $\Falg$ of the matrices
\[
\stbtmat{\alpha 1_6}{0}{0}{\alpha^{-1} 1_6} \text{\ for $\alpha \in F_{{\mathrm{alg}}}^{\times}$.}
\]

So there is a field $K/F$ such that $[K:F] = 1$ or 2 that splits $C$.  The centralizer of $C(F)$ in $A \ot K$ is (as can be seen over $\Falg$) a product of central simple algebras $B_+ \times B_-$, where $B_+$ and $B_-$ have degree 6 and $\s$ interchanges the two factors.
Put $\delta := (1, -1)$ in $B_+ \times B_-$.  Obviously $\delta^2 = 1$ and $\s(\delta) = -\delta$, because these equations hold over $\Falg$.  It follows that $\As$ is hyperbolic over $K$.

\smallskip

This already proves that $\As$ is a quadratic extension of some $\Bt$
with unitary $K/F$ involution. Moreover, by~\ref{invts} the
discriminant algebra $\D\Bt$ has Brauer class $[A]$ or 0.  
Hence we are done except if $\D\Bt$ is split and $A$ is non split. 
If so, the algebra $B$ is split by~\cite[10.30]{KMRT}, and $\As$ is
of the special form described in Example \ref{quad.split}. 
We view $Q$ as generated by pure quaternions $\delta$ and $j$ that
anti-commute and such that $\delta^2 = d$ and $j^2 = \la$ for some
$\la \in \Fx$ such that $Q=(K,\la)$.  Let $\psi = \qform{\alpha_1,
  \alpha_2, \ldots, \alpha_6}$ be the quadratic form of discriminant 1
as in equation \eqref{quad.1}.  It naturally gives a
$\gamma$-hermitian form on $Q$.  Scaling the first basis vector (of
length $\alpha_1$) by $j$ gives an isomorphic $\gamma$-hermitian form
and changes the $\alpha_1$ in the diagonalization to $\gamma(j) \alpha_1 j = j^2 \alpha_1 = \la \alpha_1$; this form is obtained from the quadratic form $\psi' = \qform{\la \alpha_1, \alpha_2, \ldots, \alpha_6}$.  That is,
\[
\As \cong {\Ad_\psi} \ot (Q, \gamma) \cong {\Ad_{\psi'}} \ot (Q, \gamma).
\]
That is, $\As$ can also be constructed from $(B, \tau') = {\Ad_{\psi'}}\ot (K, \binv)$, for which the discriminant algebra is Brauer-equivalent to $(K, \disc \psi') = (K, \la) = Q$.
\end{proof}

\begin{cor}[Pfister's Theorem] \label{Pfister}
If $q$ is a $12$-dimensional form with trivial discriminant and Clifford invariant, then $q$ is isomorphic to $\phi \ot \psi$ for some $1$-Pfister form $\phi$ and $6$-dimensional quadratic form $\psi$ with trivial discriminant.
\end{cor}

\begin{proof} 
By Th.~\ref{complete}, $(M_{12}(F), \ad_q)$ is a quadratic extension
of $(M_6(K), \tau)$ for some quadratic \'etale $F$-algebra $K$ and
unitary involution $\tau$ with split discriminant algebra.  By Example
\ref{quad.split} and specifically \eqref{quad.1}, $q$ is similar to
the tensor product of the norm form of $K/F$ (call it $\phi$) with a 6-dimensional quadratic form $\psi$ of discriminant 1.  Replacing $\psi$ with a multiple does not change the discriminant, and it completes the proof.
\end{proof}

\begin{rmk} 
Theorem \ref{complete} says something about groups $G$ of type $E_7$ whose Tits index
(as defined in \cite{Ti:Cl}) is
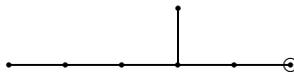
\[
\begin{picture}(6.6,1.6)
    \multiput(0.3,0.3)(1,0){6}{\circle*{\darkrad}}
    \put(3.3,1.3){\circle*{\darkrad}}

    \put(0.3,0.3){\line(1,0){5}}
    \put(3.3,1.3){\line(0,-1){1}}
    
    \put(5.3,0.3){\circle{\lrad}}
%
    \end{picture}
\]
or has more vertices circled.  For such a $G$, the obvious subdiagram of type $D_6$ corresponds to a semisimple subgroup $H$ of type $D_6$, and the isogeny class of $H$ determines $G$ by Tits's Witt-type theorem \cite[Remark 2.7.2(d)]{Ti:Cl}.  It follows from \cite[\S5]{Ti:R} that $H$ is isogenous to $\SO\As$, where $\As$ is as in Theorem \ref{complete}.

We remark that such groups $G$ are the only remaining open case of the Kneser-Tits Problem over fields of cohomological dimension $\le 2$, see \cite[8.6]{Gille:KT}.
\end{rmk}

\medskip
\noindent{\small{\textbf{Acknowledgments.}  The first author's research was partially supported by National Science Foundation grant DMS-0654502.  Both authors would like to thank the Institut des Hautes Etudes Scientifiques for a pleasant working environment while some of the research for this paper was performed.}}


\providecommand{\bysame}{\leavevmode\hbox to3em{\hrulefill}\thinspace}
\providecommand{\MR}{\relax\ifhmode\unskip\space\fi MR }
\providecommand{\MRhref}[2]{%
  \href{http://www.ams.org/mathscinet-getitem?mr=#1}{#2}
}
\providecommand{\href}[2]{#2}

\end{document}